\documentclass{article}

\usepackage{arxiv}

\usepackage[utf8]{inputenc} 
\usepackage[T1]{fontenc}    
\usepackage{url}            
\usepackage{booktabs}       
\usepackage{amsfonts}       
\usepackage{nicefrac}       
\usepackage{microtype}      
\usepackage{lipsum}
\usepackage{graphicx}
\graphicspath{ {./images/} }
\usepackage{natbib}
 \bibpunct[, ]{(}{)}{,}{a}{}{,}%

\usepackage{hyperref}       

\usepackage{amsmath,amsfonts,amssymb}
\usepackage{amsthm}
\newtheorem{definition}{Definition}
\newtheorem{example}{Example}
\newtheorem{remark}{Remark}

\newtheorem{proposition}{Proposition}
\newtheorem{assumption}{Assumption}

\usepackage{xcolor}
\usepackage{hyperref}
\usepackage{graphicx}
\usepackage{pgfplots}
\usepackage{pgfplotstable}

\usepackage{cleveref}
\crefname{algocf}{Algorithm}{Algorithms}
\Crefname{algocf}{Algorithm}{Algorithms}
\crefname{appendix}{Appendix}{Appendices}

\usepackage{tabularx,booktabs}
\usepackage{colortbl}
\usepackage{mathtools}
\usepackage{lineno}
\usepackage[textsize=scriptsize,textwidth=2.3cm]{todonotes}

\usepackage{xspace}
\usepackage[linesnumbered,ruled,vlined]{algorithm2e}

\SetCommentSty{AlgoCommentStyle}

\usepackage{tikz}
\usepackage{tikz-3dplot}
\usepackage{tikzscale}
\usetikzlibrary{matrix,chains,scopes,positioning,arrows,fit}

\usepackage{array}

\usepackage{siunitx}
\newcommand{\affected}{affected}
\newcommand{\alloptimal}{{both-solved}\xspace}
\newcommand{\difftimeouts}{{diff-timeouts}\xspace}

\newcommand{\bracket}[2]{[#1,#2]}

\newcommand{\rhs}{b}
\newcommand{\lb}[1]{\ell_{#1}}
\newcommand{\ub}[1]{u_{#1}}
\newcommand{\loclb}[1]{{\ell_{#1}^{\rho}}}
\newcommand{\locub}[1]{u_{#1}^{\rho}}
\newcommand{\slackreducing}{\textnormal{nr}}

\newcommand{\reasoncon}{C_{\textnormal{reason}}}

\newcommand{\reasoncMIR}{C_{\textnormal{cMIR}}}
\newcommand{\reasonwMIR}{C_{\textnormal{wMIR}}}
\newcommand{\contreasoncon}{C_{\textnormal{cont}}}

\newcommand{\conflictcon}{C_{\textnormal{confl}}}
\newcommand{\learnedcon}{C_{\textnormal{learn}}}
\newcommand{\state}[1]{\rho_{#1}}
\newcommand{\maxact}[2]{\textnormal{act}^{\text{max}}_{#1}({#2})}
\newcommand{\minact}[2]{\textnormal{act}^{\text{min}}_{#1}({#2})}

\newcommand{\NN}{\mathcal{N}}

\newcommand{\R}{\mathbb{R}}
\newcommand{\Rpos}{\mathbb{R}_{\ge 0}}

\newcommand{\Z}{\mathbb{Z}}
\newcommand{\Zpos}{\mathbb{Z}_{\ge 0}}

\newcommand{\I}{\mathcal{I}}
\newcommand{\cidx}{i} 
\newcommand{\vidx}{j} 
\newcommand{\nidx}{p} 
\newcommand{\bidx}{q} 
\newcommand{\ridx}{r} 

\newcommand{\var}[1]{x_{#1}}
\newcommand{\nvar}[1]{\overline{x}_{#1}}

\newcommand{\varr}{x_{r}}

\newcommand{\ie}{i.e.,\xspace}
\newcommand{\eg}{e.g.,\xspace}

\newcommand{\solver}[1]{\textsc{#1}\xspace}

\newcommand{\scip}{\solver{SCIP}}

\newcommand{\miplib}{\textsc{MIPLIB}\xspace}

\newcommand{\floor}[1]{\lfloor#1\rfloor}
\newcommand{\ceil}[1]{\lceil#1\rceil}

\newcommand\vceil[1]{\left\lceil#1\right\rceil}

\definecolor{tabcolor}{HTML}{6666AA}
\definecolor{f1}{HTML}{000060}
\definecolor{f2}{HTML}{0000FF}
\definecolor{f3}{HTML}{36648B}
\definecolor{f4}{HTML}{4682B4}
\definecolor{f5}{HTML}{5CACEE}
\definecolor{f6}{HTML}{00FFFF}
\definecolor{f7}{HTML}{00DD99}
\definecolor{f8}{HTML}{008888}
\definecolor{f9}{HTML}{000000}

\newcommand{\bigsum}[1]{\sum\limits_{#1}}
\usepackage{xcolor}
\definecolor{mygreen}{rgb}{0,0.6,0}
\newcommand{\bluediff}[1]{\textcolor{black}{#1}}

\title{Cut-based Conﬂict Analysis in Mixed Integer Programming}

\author{
 Gioni Mexi \\
  Zuse Institute Berlin, Germany\\
  \texttt{mexi@zib.de} \\
   \And
  Felipe Serrano \\
  COPT GmbH, Berlin, Germany\\
  \And
 Timo Berthold \\
  Fair Isaac Deutschland GmbH\\
  TU Berlin, Germany\\
  \And
 Ambros Gleixner \\
  HTW Berlin, Germany\\
  Zuse Institute Berlin, Germany \\
  \And
 Jakob Nordström \\
  University of Copenhagen, Denmark\\
  Lund University, Sweden\\
}

\begin{document}
\maketitle
\begin{abstract}
For almost two decades, mixed integer programming (MIP) solvers have
used graph-based conflict analysis to learn from local
infeasibilities during
branch-and-bound search.
In this paper, we improve MIP conflict analysis by instead using
reasoning based on cuts, inspired by the development of
conflict-driven solvers for pseudo-Boolean optimization. 
Phrased in MIP terminology, this type of conflict analysis can be
understood as a sequence of linear combinations, integer roundings,
and cut generation.
We leverage this MIP perspective to design
a new conflict analysis algorithm based on 
mixed integer rounding cuts, which theoretically dominates the state-of-the-art
method in pseudo-Boolean optimization
using
Chv\'atal-Gomory cuts.
Furthermore, we extend this
cut-based conflict analysis
from pure binary programs 
to mixed binary programs and---in limited form---to general
MIP with also integer-valued variables.
We perform an empirical evaluation of cut-based conflict analysis as
implemented in the open-source MIP solver \solver{SCIP},
testing it on a large and diverse set of MIP instances from \miplib2017.
Our experimental results indicate that the new algorithm improves the
default performance of \scip in terms of running time, number of nodes
in the search tree, and the number of instances solved.
 
\end{abstract}

\section{Introduction}
\label{sec:introduction}

The use of conflict analysis has a decades-old history in fields
like computer-aided verification
\citep{stallman1977forward}
and Boolean satisfiability (SAT) solving
\citep{BS97UsingCSP,MS99Grasp,MMZZM01Engineering}.  
Different sets of researchers  have independently proposed different
ways of incorporating such conflict analysis techniques into mixed
integer programming (MIP)
\citep{%
  Achterberg07ConflictAnalysis,%
  davey2002efficient,%
  sandholm2006nogood}.
We refer the reader to
\citep{witzig2021computational}
for a recent overview of techniques.

When a SAT or MIP solver encounters an infeasible subproblem during
search, the conflict analysis algorithm can be viewed in terms of a
\emph{conflict graph}, a directed acyclic graph that captures the
sequence of (branching) decisions and (propagated) deductions that led
to the infeasibility.
Using this graph, a valid constraint can be inferred by
identifying a subset of bound changes that separate the decisions
from the node where contradiction was reached.
Such learned constraints are disjunctive constraints (referred to as
clauses or clausal constraints in SAT) that are implied by the
original problem and are hence globally valid.
SAT conflict analysis can also be described in terms of
using the 
\emph{resolution} proof system
\citep{%
  Blake37Thesis,%
  DP60ComputingProcedure,%
  DLL62MachineProgram,%
  Robinson65Machine-oriented%
}
to produce syntactic derivations of learned constraints from the
input formula~\citep{BKS04TowardsUnderstanding}.
Though this perspective looks quite different, it is equivalent to the
graph-based view.


Although the use of conflict analysis in so-called
\emph{conflict-driven clause learning (CDCL)}
SAT solvers has been hugely successful, 
one drawback is that from a mathematical point of view the resolution
proof system on which it is based is quite weak,
and is known to require proofs of exponential length even for 
simple combinatorial
principles~\citep{Haken85Intractability,Urquhart87HardExamples}.
The requirement to encode the input in 
conjunctive normal form (CNF) as a collection of disjunctive clauses
incurs a further loss in expressive power.
It has therefore been studied how to lift SAT-based conflict-driven
methods to richer input formats such as 
\emph{(linear) pseudo-Boolean (PB) constraints},
which translated from SAT to MIP is just another name for
\mbox{$0$--$1$} integer linear programs with integer coefficients.
Crucially, here it turns out that the two equivalent ways of
describing SAT conflict analysis discussed above generalize in
different directions.

Many pseudo-Boolean solvers work by translating the input to CNF,
possibly by introducing auxiliary variables,
and then perform search and conflict analysis on this representation
\citep{%
  ES06TranslatingPB,%
  MML14Open-WBO,%
  JMM15GeneralizedTotalizer,%
  SN15Construction}.
Except for any auxiliary variables, this works the same as the
\emph{graph-based conflict analysis} hitherto used in MIP solvers~\cite[]{Achterberg07ConflictAnalysis}.
Another approach, however, is to adopt the derivation-based view,
but to perform the conflict analysis derivation in a proof system adapted to
\mbox{$0$--$1$} linear inequalities.
The natural candidate for such a proof system is
\emph{cutting  planes}~\citep{CCT87ComplexityCP},
which has been extensively studied in
the area of computational complexity theory.
It is a priori not obvious what it would mean to perform conflict analysis
in the cutting planes proof system, 
but methods to do so have been designed in
\citep{%
  DG02InferenceMethods,%
  CK05FastPseudoBoolean,%
  SS06Pueblo}
and are currently used in the pseudo-Boolean solvers
\solver{Sat4j}~\citep{LP10Sat4j}
and
\solver{RoundingSat}~\citep{EN18RoundingSat}.
As a theoretical method of reasoning, such
\emph{cut-based} conflict analysis turns out to be
exponentially more powerful than the graph-based conflict analysis
yielding resolution proofs.

%
To understand the difference between the two types of conflict analysis
just discussed, it is important to note that graph-based
conflict analysis does not operate on the linear constraints of the input
problem, but instead of clauses extracted from these linear constraints by
translating implications in the conflict graph to clausal constraints. 
In marked contrast, cut-based conflict analysis acts directly on the linear
constraints and performs syntactic manipulations on them using derivation
rules in the cutting planes proof system.
When a conflict is encountered during search in the form of a violated
\emph{conflict constraint}, i.e., a constraint
that detected infeasibility within the local bounds,
the constraint that propagated the last bound change leading to
the violation is identified. The goal is then to take a linear combination
of this
\emph{reason constraint}
and the conflict constraint in such a way that the variable whose bound was
propagated is eliminated, but so that the new constraint is still violated
(even with the bound change for the propagated variable removed).
In order for this to work, the reason constraint might need to be modified
before the linear combination between the reason and conflict constraints
is generated. The bound change propagated by the reason constraint might
have exploited integrality, and if so a linear combination that cancels the
propagated variable can become feasible. The way to deal with this in
pseudo-Boolean conflict analysis is that a so-called
\emph{division} or \emph{saturation} rule is applied on the reason
constraint to make the bound propagation tight even when considered over
the reals.
This is referred to as
\emph{reduction}
of the reason constraint in pseudo-Boolean terminology, whereas in MIP
language this is nothing other than a cut applied to the reason
constraint.
Since the linear combination is violated by the current set of bound
changes, we can repeat this process again,
until we derive a constraint for which a termination criterion analogous to the
\emph{unique implication point (UIP)} notion used in graph-based conflict
analysis applies.
The modified reason constraints used in the conflict analysis, together
with the conflict constraint and the branching decisions, form an
infeasible linear program (LP) even when relaxed to real values, and the
learned constraint can be viewed as a Farkas certificate for this LP
relaxation. 


For a detailed discussion of conflict analysis in mixed integer programming
and pseudo-Boolean optimization with their commonalities and differences,,
we refer the reader to Sections 1.1 and 1.2 of the conference
paper~\citep{MBGN23ImprovingConflictAnalysisMIP}
preceding the present work.
An in-depth discussion of SAT and pseudo-Boolean conflict analysis can be
found
in~\citep{BN21ProofCplxSAT},
and for a comprehensive description of  MIP solving we refer the reader to,
\eg~\citep{achterberg2007constraint}.

We remark for completeness that there is a third way of learning from local
infeasibilities in MIP called
\emph{dual-proof  analysis}~\citep{witzig2021computational}.
It is conceptually different from both graph- and cut-based conflict
analysis, in that
(i)~it can only be applied when the infeasibility has been detected via an
infeasible LP relaxation, not by propagation, and
(ii)~the learned
\emph{dual-proof constraint}
does not depend on the history of bound changes but is constructed in a
single step.
A notable commonality with cut-based conflict analysis is that dual-proof
constraints and cut-based learned constraints are derived by
aggregating a subset of the original input constraints, and so both
methods operate  syntactically on linear constraints.
One important  difference is that in cut-based conflict analysis, the set of
constraints to aggregate and the multipliers to use are computed based on
the propagations that have been implied, and also that cuts can be applied
to the individual constraints before aggregation.
In dual-proof analysis, no cuts are needed since the LP relaxation is
already infeasible, and the multipliers are obtained from the dual solution
for the LP relaxation.
Pseudo-Boolean solving has also been combined with LP solving of
relaxed versions of the input problem, and there dual-proof analysis
has been combined with pseudo-Boolean conflict analysis~%
\citep{DGN21LearnToRelax}.

\subsection{Questions Studied in This Work
  and Our Contributions}

Our focus in this work is on how cut-based
conflict analysis
as described above 
can be integrated in MIP solvers.  A first step was
taken in \citep{MBGN23ImprovingConflictAnalysisMIP} by showing how
to apply this method
for pure binary programs.
In this paper, we go much further by applying cut-based conflict
analysis in general MIP solving. We design and implement a new conflict
analysis method that is guaranteed to work not only for
\mbox{$0$--$1$} integer linear programs but also in the presence of
continuous variables. We also extend this approach to general integer
variables, but here the method is not always guaranteed to work when
there are integer variable bound changes in the 
proper interior of the variable's domain.

An important first step is to understand pseudo-Boolean conflict
analysis described in MIP terminology. We have already hinted at such
a description above, but the key is to reinterpret the reason
reduction algorithm in pseudo-Boolean conflict analysis as finding a
cutting plane that separates a fractional point from the feasible
region defined by the reason constraint.  Armed with this perspective,
we consider different reduction algorithms for pure binary
programs, and present a new reduction method using mixed integer
rounding (MIR) cuts after selectively complementing variables.
We study how different
reduction algorithms compare, and prove, in particular,
that our new MIR-based reduction algorithm provides a stronger reduced
reason constraint than the methods used in
\citep{EN18RoundingSat}
as well in the conference version
\citep{MBGN23ImprovingConflictAnalysisMIP}
preceding this paper.

Our second main contribution is for mixed binary programs.  We show
that a naive extension of cut-based conflict analysis fails for
programs with a mix of binary and real-valued variables, even if all
branching decisions are made over the binary variables.  We then
consider a more elaborate reduction algorithm, that utilizes not just
the current reason constraint to be reduced but the full history of
previous propagations and their reason constraints---in particular,
the reasons leading to bound changes for continuous variables---and
show how this algorithm produces reduced constraints that guarantee
that the conflict analysis will work also in the presence of
real-valued variables (as long as no branching decisions are made over
these variables).

As our final algorithmic contribution, we consider general mixed
integer linear programs with integer-valued variables. Unfortunately, our approach
for mixed binary programs does not extend to this setting, and we
explain why this is so. Instead, we present a heuristic
way of performing conflict analysis also in the presence of
integer-valued variables.

We have implemented all of these reduction and conflict analysis
methods in the MIP solver \solver{SCIP}, and present the
results of our computational study on MIP instances from
\miplib2017 with quantitative data for the performance of different
conflict analysis algorithms and also qualitative data about
properties of the learned constraints.
Our experiments indicate that our proposed conflict analysis algorithm generates
useful conflict constraints and improves the default performance of the MIP
solver \scip not only in terms of running time and number of nodes in the search
tree, but also in terms of the number of instances solved.

\subsection{Organization of This Paper}
After a review of preliminaries in
\Cref{sec:preliminaries},
we formalize our MIP interpretation of pseudo-Boolean cut-based conflict
analysis in
\Cref{sec:conflict-intuition}.
\Cref{sec:reduction-techniques-for-bp}
presents a collection of different reduction algorithms,
and 
\Cref{sec:dominance}
studies dominance relationships between them.
We extend cut-based conflict analysis to mixed binary programs in
\Cref{sec:cut-based-conflict-analysis-for-mbp},
and discuss general MIP problems with integer variables in
\Cref{sec:general-integer-variables}.
The results from our computational evaluation are presented in
\Cref{sec:experiments}, after which we provide some concluding remarks
and discuss future research directions in 
\Cref{sec:conclusion}.

\section{Preliminaries and Notation}
\label{sec:preliminaries}

In this section we provide the necessary notation and definitions used throughout the paper.
We consider \emph{mixed integer programs} (MIPs) of the form
	\begin{equation}
		\label{eq:mip}
		\begin{aligned}
			 & \underset{x\in \R^n}{\text{min}} & c^\top  x &  \\
		   & \text{s.t.}                       &    Ax &  \geq \rhs,      &  &                 \\
			 &                                  & {\ell}_{\vidx} \le x_{\vidx} & \le u_{\vidx} &  &\text{ for all } j \in \{1,\dots,n\}, \\
			 &                                  &      x_{\vidx} & \in \Z     &   &\text{ for all } j \in \I,
		\end{aligned}
	\end{equation}
where $m,n \in \Zpos$, $A\in\R^{m \times n}$, $\rhs \in \R^m$, $c \in \R^n$, $\ell,u \in (\R\cup \{\pm \infty\})^n$, and $ \I \subseteq \NN := \{1,\dots,n\}$.
A single constraint hence takes the form $\sum_{\vidx} a_{\cidx\vidx}x_{\vidx} \geq \rhs_i$, where $a_{i}^\top \in \R^n$ is the $\cidx$-th row of the matrix $A$, and $\rhs_i\in\R$ the corresponding entry of the vector $\rhs$. 
For the sake of readability, we omit the index $i$ when it is not essential to the context and define a generic constraint as 
$C: \,\sum_{\vidx} a_\vidx x_{\vidx} \geq \rhs$.
By slight abuse of notation, we use $C$ also to denote the half-space~$\{x\in\R^n:\sum_{\vidx} a_\vidx x_{\vidx} \geq \rhs\}$ defined by the constraint~$C$.

Depending on the set $\I$ and the bound vectors $\ell$ and $u$, we distinguish the following special cases of MIPs.
We call \eqref{eq:mip} a \emph{mixed binary program} (MBP) if $\ell_{\vidx} = 0$ and $u_{\vidx} = 1$ for all $j \in \I$.
If additionally $\I = \{1,\dots,n\}$, then we call \eqref{eq:mip} a \emph{pure binary program} (BP).


\subsection{States and Activities}
\label{sec:notation}
When processing a node in the branch-and-bound tree, we may encounter infeasibilities. To apply conflict analysis we need information on all bound changes of variables that were applied by branching or propagation between the root node and the current node.
At each node $\nidx$, let $\bidx=0,1,2,\ldots$ index the sequence of bound changes on some variable at this node.
By $\state{\nidx,\bidx}$ we then refer to the entire state of the problem after applying all bound changes on the unique path from the root node to node~$\nidx$, including bound changes $0,1,\ldots,q$ at node~$p$.
At a node $\nidx$, the first state $\state{\nidx,0}$ always represents the problem state after the single
branching decision taken at this node.
Between two subsequent states $\state{\nidx,\bidx}$ and $\state{\nidx,\bidx+1}$, only one bound change on a single variable is recorded, and therefore either the lower or the upper bound vector differs in exactly one entry.

In conflict analysis, we typically only consider the states on the unique path
from the root node to the current node.
In this case, we can use the so-called \emph{decision level} as node index~$p$,
which gives the number of branching decisions up to and including the current
state.
Then the states on a single path can be ordered lexicographically by their indices via
\newcommand{\lexlt}{\prec}
\newcommand{\lexleq}{\preccurlyeq}
\[
(\nidx,\bidx) \lexleq (\nidx',\bidx')
:\Leftrightarrow
\nidx < \nidx' \lor (\nidx = \nidx' \land \bidx \leq \bidx').
\]
%
In the following, we often omit the node index $\nidx$ and the bound change index $\bidx$ if the specific positions are not crucial for the argument we are trying to convey.
For each state $\rho$, we can query several relevant
pieces of  
information such as
\newcommand{\bdcvar}{\textrm{varidx}}
\newcommand{\reasoncons}{\textrm{reason}}
\begin{itemize}
  \item the vector~$\loclb{}\in(\R\cup\{-\infty\})^n$ of \emph{local lower bounds} of all variables,
  \item the vector~$\locub{}\in(\R\cup\{\infty\})^n$ of \emph{local upper bounds} of all variables,
  \item the \emph{variable index}~$\bdcvar(\state{})\in\NN$ of the corresponding bound change, and 
  \item the \emph{reason constraint} $\reasoncons(\state{})$ from which the bound change was derived.
  \end{itemize}


For a generic constraint $C: \sum_{\vidx} a_\vidx \var{\vidx} \geq
\rhs$, we denote the \emph{maximal activity} of $C$ under state
$\state{}$ as 
%
$
  \maxact{}{C,\state{}} := \sum_{\vidx}
  \max\{a_\vidx \cdot \loclb{\vidx}, \, a_\vidx \cdot \locub{\vidx}\}.
$
Similarly, we define the \emph{minimal activity} of $C$ under
$\state{}$ as 
$
\minact{}{C,\state{}}:=
\sum_{\vidx} \min\{a_\vidx \cdot \loclb{\vidx}, \, a_\vidx \cdot \locub{\vidx}\}.
$
We call the constraint $C$ \emph{infeasible} under some state if its maximal activity is less than the right-hand side, \ie if the respective inequality cannot hold for any solution within the local bounds.
Otherwise, we call $C$ \emph{feasible}.
We are particularly interested in the fact whether the bound changes on a variable affect the current maximal activity of a constraint.
We call a variable~$x_j$ \emph{relaxable} for constraint~$C$ under state~$\state{}$ if
$\maxact{}{C,\state{}}$ remains unchanged when we replace its local bounds $\loclb{\vidx}$ and
$\locub{\vidx}$ by its global bounds $\lb{\vidx}$ and $\ub{\vidx}$, respectively.
\bluediff{By definition, this is the case if and only if $a_{\vidx}=0$ or $a_{\vidx}> 0 \land \locub{\vidx}=\ub{\vidx} $ or
$a_{\vidx}< 0 \land \loclb{\vidx}=\lb{\vidx}$.}
Otherwise, the variable is called \emph{non-relaxable}.
Note that these terms provide a generalization from the PB to the MIP setting.
In the literature on pseudo-Boolean optimization, where there are
  only binary variables, 
  relaxable and non-relaxable
  variables correspond to non-falsified and falsified literals, respectively.
  
\subsection{Propagation of Linear Constraints}
Beyond identifying trivial infeasibility, the maximal activity of a constraint under some state
$\rho$ can also be used to identify whether variable bounds can be tightened.
A commonly used propagator for linear constraints is the bound strengthening technique going back to~\cite[]{brearley1975analysis}. 
If $\maxact{}{C,\state{}} < \infty$, then we can rewrite constraint $C$ as
%
\begin{align}
  \label{eq:actprop}
  a_\ridx\var{r}
  \geq \rhs - \sum_{j \neq r} a_\vidx\var{\vidx}
  \geq \rhs - \max_{\var{}\in[\loclb{},\locub{}]}\sum_{j \neq r}a_\vidx\var{\vidx}
\end{align}
For $a_\ridx > 0$ we obtain
\begin{align*}
  a_\ridx\var{r}
  \geq \rhs - \maxact{}{C,\state{}} + a_\ridx\locub{r} 
\end{align*}
and dividing by $a_\ridx$ gives 
\begin{align}
  \label{eq:proplb}
  \var{r} \geq \locub{r} + \frac{\rhs - \maxact{}{C,\state{}}}{a_\ridx} =: \tilde{\ell}_r,
\end{align}
which can be used to tighten the lower bound of $\var{r}$ if $\tilde{\ell}_r > \loclb{r}$.
In a similar fashion, if $a_\ridx < 0$, we can deduce the upper bound $\var{r} \leq \tilde{u}_r := \loclb{r} + \frac{ \rhs - \maxact{}{C,\state{}}}{a_\ridx}$ if $\tilde{u}_r < \locub{r}$.
For completeness, note that if $\maxact{}{C,\state{}}=\infty$ because $a_\ridx>0 \land \locub{r}=\infty$ or $a_\ridx<0 \land \loclb{r}=-\infty$, but $\max_{\var{}\in[\loclb{},\locub{}]}\sum_{j \neq r}a_\vidx\var{\vidx} < \infty$, then \eqref{eq:actprop} can be used directly to derive valid bounds.

For integer variables, \ie for $r\in\I$, the derived bounds can be rounded to
$\ceil{\tilde{\ell}_r}$ and $\floor{\tilde{u}_r}$, respectively.
As will become clear in \Cref{sec:conflict-intuition}, we are
particularly interested in propagations where this second step is
redundant because either the variable~$x_r$ is continuous or the
propagated bounds are already
integral, 
\ie $\tilde{\ell}_r\in\Z$ and $\tilde{u}_r\in\Z$,
respectively.
We refer to such propagations as \emph{tight}.

\subsection{Coefﬁcient Tightening and Cutting Planes}

The following techniques from the PB and MIP literature to strengthen linear inequalities are key ingredients of our conflict analysis algorithms. \emph{Coefficient tightening}~\cite[]{brearley1975analysis} seeks to tighten the coefficients of integer variables in $C$ to derive a more restrictive constraint while preserving all feasible integer solutions. 

\begin{definition}[Coefficient Tightening]
  \label{def:CoefTight}
  Let $C: \sum_{\vidx \in J \cup K} a_\vidx x_{\vidx} \geq \rhs$, $x \in \Zpos^J
  \times \Rpos^{K}$, where $J \subseteq~\I, K \subseteq \NN \setminus \I$, and $a_\vidx > 0$ for all $\vidx \in J$. Further, let $\minact{}{C}$ denote the minimal activity of
  the constraint under the global bounds. Then
  \emph{coefficient tightening} 
  yields the constraint
  \begin{equation}
    \label{eq:ct}
  \sum_{\vidx \in J} \min \{a_\vidx, \tilde b\}\var{\vidx} + \sum_{k \in K} a_k\var{k}  \geq \rhs - \sum_{\vidx \in J} \max\{0,a_\vidx- \tilde b\}\lb{\vidx},
  \end{equation}
  where $\tilde b = \rhs - \minact{}{C}$.
  The general case for $a_\vidx \in \R \setminus \{0\}, j \in J$ is analogous.
  \end{definition}

  Another well known cut from the IP literature which is already used in the context of pseudo-Boolean conflict analysis is the \emph{Chv\'atal-Gomory cut}~\cite[]{chvatal1973edmonds}.
  \begin{definition}[Chv\'atal-Gomory Cut]
      \label{def:cg}
      Let $C: \sum_{ \vidx \in J} a_\vidx \var{\vidx} \geq \rhs$ with $x_{} \in \Zpos^J$, $ J \subseteq \I$. The \emph{Chv\'atal-Gomory (CG) Cut} of $C$ is given by the constraint
      \begin{equation}
        \label{eq:cg}
        \sum_{\vidx\in J} \vceil{a_\vidx}\var{\vidx}  \geq \vceil{\rhs}.
      \end{equation}
  
  \end{definition}
  To see why \eqref{eq:cg} is valid for 
  $\{ x \in \Zpos^J \, : \,  \sum_{ \vidx \in J} a_\vidx \var{\vidx} \geq \rhs \}$, we can think of it as two steps:
  rounding up coefficients on the left-hand side relaxes the constraint and
  is hence valid;
  the validity of rounding up the right-hand side follows from the integrality of the left-hand side.
  
Next, we recall  the \emph{Mixed Integer Rounding (MIR) cut}~\cite[]{marchand2001aggregation}.
\begin{definition}[Mixed Integer Rounding Cut]
    \label{def:MIRcut}
    Let $C: \sum_{ \vidx \in J\cup K} a_\vidx \var{\vidx} \geq \rhs$ with $x \in \Zpos^J \times \Rpos^K$, where $J \subseteq \I, K \subseteq \NN \setminus \I$. 
    The \emph{Mixed Integer Rounding (MIR) Cut} of $C$ is given by 
    \begin{equation}
      \label{eq:mir}
      \sum_{ \substack{\vidx \in K:\\ a_\vidx >0}} \frac{a_\vidx}{f(\rhs)} \var{\vidx} +
   \sum_{ \vidx \in J} \left(\lfloor a_j \rfloor + \min\left\{1, \frac{f(a_j)}{f(b)}\right\}\right) x_j \geq \vceil{\rhs},
    \end{equation}
    \bluediff{where $f(r) = r - \lfloor r \rfloor$  is the fractional part of $r$.}
\end{definition}
The proof that \eqref{eq:mir} is valid for 
$\{x \in \Zpos^J \times \Rpos^K \, : \,  \sum_{ \vidx } a_\vidx \var{\vidx} \geq \rhs \}$
can found in
\cite[]{marchand2001aggregation}.


  \subsection{Weakening and Complementation}
  Two operations on constraints that are used in later sections are \textit{weakening} and \textit{complementation} of variables.
  Weakening is a valid operation since it simply adds a multiple of the globally valid bound constraints 
  $\var{\vidx} \le \ub{\vidx} \Leftrightarrow -\var{\vidx} \ge -\ub{\vidx}$, or $\var{\vidx} \geq \lb{\vidx}$ to~$C$.
  Weakening a variable $\var{s}$ in a constraint $C: \,\sum_{\vidx}a_\vidx\var{\vidx} \geq \rhs$ is defined as:
  $
  \text{weaken}(C,\, \var{s}) := \sum_{\vidx \neq s} a_\vidx\var{\vidx} \geq \rhs - \text{max}\{ a_s \ub{s}, a_s \lb{s} \}.
  $
  Weakening entails a loss of information. However, as we will see, it is a necessary operation in some reduction algorithms. 
  Note that whenever weakening is applied on relaxable variables at the current
  state, it does not change the bound propagations of the remaining variables in the constraint.
  
  The second operation that we use in the reduction is complementation of variables.
  Complementation is a valid operation that does not entail a loss of information since it simply replaces a variable $\var{s}$ by its complement $\nvar{s} = \ub{s}- \var{s}$. 
  Complementing a variable $\var{s}$ in the general constraint $C$ from above is defined as:
  $
      \text{complement}(C,\, \var{s}) := \sum_{\vidx \neq s} a_\vidx\var{\vidx} - a_s \nvar{s}\geq \rhs - a_s \ub{s}.
  $

\section{The Simple Case: Conflict Analysis under Tight Propagations}
\label{sec:conflict-intuition}

Conflict analysis is applicable whenever a locally infeasible constraint is detected at a node of the branch-and-bound search.
In MIP solvers, this constraint may be detected directly during propagation or as trivially infeasible aggregation of constraints resulting from an infeasible LP relaxation~\citep{Achterberg07ConflictAnalysis}.
This can occur after many propagations on the current node,
and the main goal of conflict analysis is to obtain a constraint that would have identified the infeasibility earlier.
Such a constraint would have propagated in an earlier node, preventing us from visiting this infeasible node in the first place.%
The technique for constructing such a constraint is as follows. 

Given the constraint that found the infeasibility, called \emph{conflict constraint}, and the state history, 
we can find the last constraint that propagated a variable at a state $\rho$ which contributed to the infeasibility of the conflict constraint, called \emph{reason constraint}.
The conflict and reason constraints form a system of \emph{two} constraints that cannot be satisfied simultaneously at a state $\tilde \rho$ before $\rho$.
Now, the goal is to obtain a \emph{single} globally valid constraint that is infeasible under $\tilde \rho$.
If we have a general recipe to achieve this, this process can be repeated, using the learned constraint as the new conflict constraint.
A common criterion is to iterate this step until a so-called \emph{first unique implication point} (FUIP) is reached, i.e., until the learned constraint would propagate some variable at a state in the previous decision level. 
Note that the learned constraint may also be infeasible under the global bounds, in which case we have proven that the problem is globally infeasible.
To summarize, the fundamental problem of conflict analysis is:
Given two constraints $\conflictcon$ and $\reasoncon$, a local domain $L$, and a global domain $G$, such that
the system $\{\conflictcon, \reasoncon, L\}$ is infeasible, find a constraint $\learnedcon$ that is valid for $\{\conflictcon, \reasoncon, G\}$
such that $\{\learnedcon, L\}$ is infeasible.

In SAT, this fundamental problem is easy to solve.
Indeed, constraints in SAT are of the form $\sum_{i \in J_0} x_i + \sum_{j \in J_1} \bar x_j \geq 1$, where $\bar x_j$ is the negation of $x_j$, \ie $\bar x_j = 1 - x_j$.
The only way such a constraint can propagate is that all variables but one are fixed to 0.
Without loss of generality, let us assume that $x_r$ is the variable propagated by $\reasoncon : x_r + \sum_{i \in J_0} x_i + \sum_{j \in J_1} \bar x_j \geq 1$ in the local domain $L$.
Then, in $L$ it must be that $x_i = 0$  for all $i \in J_0$ and $x_j = 1$ for all $j \in J_1$.
Furthermore, since $\conflictcon$ is infeasible after the bound tightening $x_r = 1$, it must be of the form
$\bar x_r + \sum_{i \in J_0'} x_i + \sum_{j \in J_1'} \bar x_j \geq 1$
with $x_i = 0$ for all $i \in J_0'$ and $x_j = 1$ for all $j \in J_1'$ in $L$.
The learned constraint $\learnedcon$ is then obtained by the so-called \emph{resolution rule}, 
which is nothing else than taking the sum of both constraints,
\begin{equation} \label{eq:theSum}
    \sum_{i \in J_0} x_i + \sum_{i \in J_0'} x_i  + \sum_{j \in J_1 } \bar x_j +\sum_{j \in J_1'} \bar x_j \geq 1.
\end{equation}
and applying coefficient tightening to \eqref{eq:theSum} which yields 
\begin{equation} \label{eq:theSum2}
    \sum_{i \in J_0 \cup J_0'} x_i + \sum_{j \in J_1 \cup J_1'} \bar x_j \geq 1.
\end{equation}

This constraint is clearly valid in any global domain $G$ and infeasible in $L$ and also corresponds to the constraint obtained by resolution in graph-based conflict analysis \cite[]{Achterberg07ConflictAnalysis}.
%
%
Note that the above argument didn't use integrality information.
\begin{remark}
We can interpret the above technique as trying to eliminate $x_r$ from the system  $\{\conflictcon, \reasoncon, L\}$ by Fourier-Motzkin elimination.
Indeed, the system is still infeasible even when we relax the domain of $x_r$ to $(-\infty, \infty)$.
Then, by Fourier-Motzkin elimination \cite[]{Williams76FourierMotzkin}, projecting out $x_r$ yields exactly \eqref{eq:theSum}
Note that this immediately proves the infeasibility of
\eqref{eq:theSum} in $L$, since the projection of the empty set is the
empty set.
\end{remark}
The idea above can be generalized to linear constraints. 
The linear combination of the constraints that eliminates a given variable followed by coefficient tightening is what is known as \emph{generalized resolution}~\cite[]{Hooker88Generalized,Hooker92Generalized}.
In the context of conflict analysis, the learned constraint is also called the \emph{resolvent}.
As the following example shows, the naive application of generalized resolution does not necessarily guarantee that the resolvent is infeasible in the local domain.

\begin{example}
   \label{ex:noconflict}
   Consider the two pure binary constraints $C_r : x_1 + x_2 +2x_3 \ge 2$,  $C_c : x_1 - 2x_3 + x_4 + x_5 \ge 1$, 
   $L = \{0\} \times \{0,1\}^4$, and $G = \{0,1\}^5$.
   Under $L$, constraint~$C_r$ propagates $x_3 \geq 0.5$. Since $x_3$ is binary, we conclude that $x_3 \geq 1$.
   Then, $C_c$ is clearly infeasible.
   Eliminating $x_3$, by adding the two constraints and applying coefficient tightening, we obtain
   the constraint $2x_1 + \var{2} + \var{4} + \var{5} \ge 3$ which
   is globally valid, but not infeasible under $L$.
\end{example}

A notable difference between the SAT case and Example~\ref{ex:noconflict} is
that in Example~\ref{ex:noconflict} we had to use integrality information
because $\reasoncon$ did not propagate $x_3$ \emph{tightly}, i.e., to an integer value.
Indeed, this is the only reason why generalized resolution can fail.
The following proposition shows that whenever no integrality information is used during the propagation of $\reasoncon$, the 
resolvent of $\reasoncon$ and $\conflictcon$ that eliminates $\var{\ridx}$ remains infeasible.

\begin{proposition}
    \label{prop:tight-prop-resolution}
        Let $\reasoncon: a_\ridx \var{\ridx}+ \sum_{ \vidx \neq \ridx } a_\vidx
        \var{\vidx} \geq \rhs$ be a constraint propagating a variable $\var{\ridx}$ in the state $\state{}$ tightly.
        Further, assume that $\conflictcon: a'_{\ridx}\var{\ridx} + \sum_{ \vidx \neq \ridx } a'_{\vidx}\var{\vidx} \geq b'$ becomes infeasible in the state $\state{}$. 
        Then the resolvent of $\reasoncon$ and $\conflictcon$ that eliminates $\var{\ridx}$ remains infeasible in the state $\state{}$.
\end{proposition}
\begin{proof}
    Since $a_\ridx$ and $a'_{\ridx}$ are both non-zero and have opposite signs,
    without loss of generality, we assume that $a_\ridx = 1$ and $a'_{\ridx} = -1$.
    Then, the constraint $\reasoncon$ propagates a lower bound $\beta$ on the variable $\var{\ridx}$, namely
    \[
         \var{\ridx}
         \geq \min_{x\in[\loclb{},\locub{}]}\{b - {\sum_{ \vidx \neq \ridx } a_\vidx \var{\vidx}}\}
         = b - \max_{x\in[\loclb{},\locub{}]}{\sum_{ \vidx \neq \ridx} a_\vidx \var{\vidx}}
         =: \beta.
    \]
    Now, since $\conflictcon$ is infeasible in the state $\state{}$, we have that
    \[ 
        -\beta + \max_{x\in[\loclb{},\locub{}]}\sum_{ \vidx \neq \ridx } a'_{\vidx} \var{\vidx} < b'
    \]
    The right-hand side of the resolvent $C_{\text{res}} := \reasoncon + \conflictcon$ is $b + b'$.
    However, its maximum activity is
    \begin{align*}
        \max_{x\in[\loclb{},\locub{}]}{\sum_{ \vidx \neq \ridx } (a_\vidx + a'_{\vidx}) \var{\vidx}}
        & \leq \max_{x\in[\loclb{},\locub{}]}{\sum_{ \vidx \neq \ridx} a_\vidx \var{\vidx}}
        + \max_{x\in[\loclb{},\locub{}]}\sum_{ \vidx \neq \ridx } a'_{\vidx} \var{\vidx} \\  
        &  = b - \beta + \max_{x\in[\loclb{},\locub{}]}\sum_{ \vidx \neq \ridx } a'_{\vidx} \var{\vidx} < b + b',
    \end{align*}
    which shows that $C_{\text{res}}$ is infeasible in the state $\state{}$.
\end{proof} 
To obtain a less syntactic, more geometric perspective, let us view the lower bound change on a variable $x_r$ from propagation of $ \reasoncon: a_\ridx \var{\ridx}+ \sum_{ \vidx \neq \ridx } a_\vidx
        \var{\vidx} \geq \rhs$, $a_r > 0$  as the solution of the one-row linear program
\begin{equation} 
    \label{eq:1-row}
    \min \{x_r \, : \, a_r x_r + \sum_{\vidx \neq r } a_\vidx x_\vidx \geq \rhs, \, x_\vidx \in [\loclb{\vidx}, \locub{\vidx}] \}.
\end{equation}
Non-tight propagation can occur if and only if the optimum of \eqref{eq:1-row}
over the current state is attained at a
non-integer vertex, as illustrated in~\Cref{fig:nonintegervertex} for the reason constraint of \Cref{ex:noconflict}.
By contrast, for SAT constraints the feasible region defined by the reason
constraint and the global domain does not contain non-integer vertices.
This geometric perspective also shows a clear path forward to make conflict analysis work:
by strengthening the propagating reason constraint~$\reasoncon$ in order to cut off the non-integer vertex from the feasible region $\reasoncon \cap G$.
Note that for pure binary programs the reason is a knapsack constraint, and hence it is sufficient to
study cuts for the knapsack polytope~\cite[]{hojny2020knapsack}.
In the PB literature, the techniques to achieve such cuts are called
\emph{reduction techniques}.
In the next section, we first present the cut-based conflict analysis algorithm and then
proceed to discuss various reduction techniques.
\begin{figure}[h]
    \centering
    \begin{tikzpicture}
        \begin{axis}[
            axis lines = left,
            xlabel = {$x_2$},
            ylabel = {$x_3$},
            xmin = 0, xmax = 1,
            ymin = 0, ymax = 1,
            xtick = {0, 0.5, 1},
            ytick = {0, 0.5, 1},
            tick label style={font=\small},
            domain = 0:1,
            samples = 100,
            width = 5cm,
            height = 5cm,
        ]
    
        \addplot [
            fill=blue!20, 
            draw=none,
            domain=0:1,
        ] { (2 - x) / 2 } |- (axis cs:1,1) |- (axis cs:0,1) \closedcycle;

        \addplot [
            thick, 
            blue,
        ] { (2 - x) / 2 };
        
        \addplot[
            only marks,
            mark=*,
            mark size=2pt,
            color=red
        ]
        coordinates {
            (1, 0.5)
        };
    
        \end{axis}
    \end{tikzpicture}
    \caption{Inequality $x_1 + x_2 + 2x_3 \geq 2$ on the face of the polytope with $x_1 = 0$. The fractional vertex is $(0, 1, 0.5)$.}
    \label{fig:nonintegervertex}
\end{figure}
\section{Reduction Techniques for Pure Binary Programs}
\label{sec:reduction-techniques-for-bp}

\Cref{algo:CA} shows the base algorithm for all variants of cut-based conflict analysis considered in this paper.
The algorithm is initialized with an infeasible state $\state{\nidx,\bidx}$ and a conflicting constraint
$\conflictcon$ in $\state{\nidx,\bidx}$.
First, the learned conflict constraint $\learnedcon$ is set to the conflict constraint $\conflictcon$.
In each iteration, the state $\state{\nidx,\bidx}$ is 
set to the smallest state, with respect to the lexicographic order, such that the learned conflict constraint is infeasible,
and we extract the variable $\varr$ whose bound was changed in $\state{\nidx,\bidx}$. 
If the bound change was due to propagation of a constraint, then we extract the reason 
constraint $\reasoncon$ that propagated~$\varr$. In line \ref{line:reduce} we ``reduce'' 
the reason constraint such that the resolvent, \ie linear combination, of $\learnedcon$ and the reduced reason
$\reasoncon$ (Line \ref{line:resolve}) that cancels $\varr$ remains infeasible.
The learned conflict constraint is set to the resolvent, 
which can then be strengthened in Line \ref{line:strengthen} by, e.g., applying coefficient tightening as in SAT resolution conflict analysis.
We continue this process until we reach an
FUIP ($\learnedcon$ is \emph{asserting}) or $\learnedcon$ proves global infeasibility.
\newcommand{\reduce}{\textnormal{reduce}}
\newcommand{\resolve}{\textnormal{resolve}}
\newcommand{\strengthen}{\textnormal{strengthen}}
\begin{algorithm}[h]

    \DontPrintSemicolon
	\SetKwInOut{Input}{Input}\SetKwInOut{Output}{Output}
	\SetKwInOut{Init}{Initialization}
    \SetArgSty{upshape}
    \SetKwFunction{pop}{pop}
    \SetKwFunction{reason}{reason}
	\Input{initial conflict constraint $\conflictcon$, infeasible state $\state{\nidx,\bidx}$}
	\Output{learned conflict constraint $\learnedcon$}
	$\learnedcon \leftarrow \conflictcon$ \;
    \While{$\learnedcon$ not asserting \textbf{and} $\learnedcon \neq  \, \perp$\label{line:startloopCA}}
	{
        $( \nidx,\bidx ) \leftarrow \min \{ (\tilde{\nidx},\tilde{\bidx}) \, | \, \learnedcon \text{ is infeasible in } \state{\tilde{\nidx},\tilde{\bidx}}\} $\;
        $ \ridx\leftarrow \bdcvar(\state{\nidx,\bidx})$\;
        \If{ $\var{\ridx}$ propagated }
        {
            $\reasoncon \leftarrow \reasoncons(\state{\nidx,\bidx} )$\;
            $\reasoncon \leftarrow \reduce(\reasoncon,\learnedcon, \state{\nidx,\bidx})$\label{line:reduce}\;
            $\learnedcon \leftarrow \resolve(\learnedcon,\reasoncon,\varr)$\label{line:resolve}\;
            $\learnedcon \leftarrow \strengthen(\learnedcon)$\label{line:strengthen} \;
        }
   }
 	\Return{$\learnedcon$}\;
	\caption{Cut-Based Conflict Analysis for 0-1 IP}
    \label{algo:CA}
\end{algorithm}
The conflict analysis algorithm can fail if the resolvent becomes feasible.
As explained in the previous section, this can occur only if the reason
constraint $\reasoncon$ does not propagate tightly.
In the following, we present various reduction algorithms for
the reason constraint $\reasoncon$ that ensure that this constraint propagates tightly enough:
the SAT-like \emph{clausal-based reduction}~\cite{Achterberg07ConflictAnalysis}, which reduces the reason
constraint to a clause responsible for the propagation; the \emph{coefficient
tightening-based reduction}, which is a generalization of \cite[]{CK05FastPseudoBoolean}; and a stronger version 
of the \emph{MIR-based reduction} presented in~\cite[]{MBGN23ImprovingConflictAnalysisMIP}.

As we mentioned at the end of last section, the goal of reduction algorithms
is to obtain a constraint that propagates a variable tightly from a reason constraint
that did not propagate tightly.
Therefore, reduction algorithms only look at a single given constraint and
the current domain in a state $\state{}$.
Hence, in this section we work under the following assumption.
\begin{assumption} \label{ass:reasonForm}
    The reason constraint is of the form 
    $
        \reasoncon: x_r + \sum_{ j \in J } a_j x_j \geq b,
    $
    where $r\not\in J \subseteq \I$ and $a_j \geq 0$ for all $j \in J$.
    Furthermore, $\reasoncon$ propagates $x_r$ in $\state{}$ non-tightly, \ie 
    $ b -   \max \sum_{ j \in J } a_j x_j  \, \not\in\, \Z. $
\end{assumption}
This assumption can be made without loss of generality.
%
Indeed, from a general reason constraint we can first complement any variable with a negative coefficient.
Afterwards, we can divide by the coefficient of the variable $x_r$ to obtain the aforementioned form.
Obviously, the information in both constraints is exactly the same.

\subsection{Coefficient Tightening-based Reduction}
\label{sec:CoefficientTightening}
\cref{algo:ReduceSAT} summarizes the coefficient tightening-based reduction. 
Similar to the implementation in a PB solver \cite[]{CK05FastPseudoBoolean}, in each iteration, the algorithm picks a
relaxable variable in the reason constraint different from the variable we are resolving on and weakens it.
Then it applies coefficient tightening to the
resulting constraint. 
After each iteration, the algorithm checks if the resolvent of the reason constraint and conflict constraint is infeasible.
In this case, the algorithm terminates and returns the reduced reason constraint.
A proof that the reduction algorithm produces a reason constraint that
propagates strongly enough, in that the reduced reason is guaranteed to
propagate tightly after all relaxable variables have been weakened, can be
found in~\citep{MBGN23ImprovingConflictAnalysisMIP}.
\newcommand{\weaken}{\textnormal{weaken}}
\newcommand{\coefTight}{\textnormal{coefTight}}
\begin{algorithm}[h]
    \DontPrintSemicolon
	\SetKwInOut{Input}{Input}\SetKwInOut{Output}{Output}
	\SetKwInOut{Init}{Initialization}
	\Input{conflict constraint $\conflictcon$, reason constraint $\reasoncon$, variable to resolve $\varr$, state $\state{}$}
	\Output{reduced reason $\reasoncon$}
   \While{$\resolve(\reasoncon,\conflictcon,\varr)$ is feasible in $\state{}$}
    {
        $\var{s} \leftarrow$ relaxable variable in $ \reasoncon\backslash \{\varr\}$ \;
        $\reasoncon \leftarrow \weaken(\reasoncon,\var{s})$ \;
      
        $\reasoncon \leftarrow 
        \coefTight(\reasoncon)$ \;
    }
 	\Return{$\reasoncon$}\;

	\caption{Coefficient Tightening-based Reduction Algorithm \label{algo:ReduceSAT}}
\end{algorithm}

\subsection{cMIR-based Reduction}
\label{sec:MIR}
For pure binary constraints, a very competitive alternative to coefficient tightening in the reduction algorithm is based on Chvátal-Gomory cuts \cite[]{EN18RoundingSat}.
In this reduction, relaxable variables with fractional coefficients in the reason constraint are weakened before 
applying Chvátal-Gomory rounding  as in \eqref{def:cg}. In \cite[]{MBGN23ImprovingConflictAnalysisMIP} we show that applying the more general MIR cut instead 
yields a reduced reason constraint that is at least as strong.

\newcommand{\poscompl}{P}
\newcommand{\negcompl}{N}
\newcommand{\poscoef}{a_\vidx > 0}
\newcommand{\negcoef}{a_\vidx < 0}
\newcommand{\posnocompl}{\substack{\vidx: \poscoef, \\ \vidx \notin \poscompl}}
\newcommand{\negnocompl}{\substack{\vidx: \negcoef, \\ \vidx \notin \negcompl}}
\newcommand{\posnocomplforslack}{\substack{\vidx: \poscoef, \\ \vidx \notin \poscompl, \locub{\vidx} = 1}}
\newcommand{\negnocomplforslack}{\substack{\vidx: \negcoef, \\ \vidx \notin \negcompl, \loclb{\vidx} = 1}}
\newcommand{\unchangedslack}{S}

Next, we present a further improved reduction technique also based on the MIR formula
\eqref{def:MIRcut}.
The improvement comes from the fact that weakening becomes obsolete after complementing relaxable variables.
We call this reduction \emph{cMIR-based reduction}.
In \Cref{sec:dominance} we show one of the main results of this paper, which is the dominance of cMIR-based reduction 
over both Chvátal-Gomory and MIR-based reduction from \cite[]{MBGN23ImprovingConflictAnalysisMIP}.
But first, we show that the reduced reason
constraint from the cMIR-based reduction propagates $x_r$ tightly.
\begin{proposition} \label{prop:mirReduction}
    Let $\reasoncon$ be as in Assumption \ref{ass:reasonForm}.     
    Complementing all  variables in $P = \{ j \in J \, : \, \locub{j} = 1 \} $ 
    and applying MIR gives the reduced reason constraint
    \begin{equation} \label{eq:cMIR}
     \reasoncMIR : x_r + \sum_{j \notin P} \psi(a_j) x_j - \sum_{j \in P} \psi(-a_j) x_j \geq 1 - \sum_{j \in P} \psi(-a_j) 
    \end{equation}
    with
    \[ 
        \psi(a) = \lfloor a \rfloor + \min\left\{1, \frac{f(a)}{f(b - \sum_{j \in P} a_j)}\right\}.
    \]
    $\reasoncMIR$ propagates $x_r$ tightly.
\end{proposition}
\begin{proof}
Complementing the variables in $P$ yields
    \begin{align}
        \label{eq:ComplementBeforeMIR}
        \var{r}
        + \bigsum{\vidx \notin P} a_\vidx \var{\vidx}
        - \bigsum{\vidx \in P}a_\vidx \nvar{\vidx} \geq \rhs 
        - \bigsum{\vidx \in P} a_\vidx =: \tilde b.
    \end{align}
    Note that $0 < \tilde b < 1$ by Assumption 1
    Indeed, since $\reasoncon$ propagates non-tightly, we have that
    $\min \{ b - \sum_{j \in J} a_j x_j \}$ is between 0 and 1.
        After applying MIR to the complemented constraint we obtain
        \begin{equation}
            \reasoncMIR : x_r + \sum_{j \notin P} \psi(a_j) x_j + \sum_{j \in P} \psi(-a_j) \bar x_j \geq 1 
        \end{equation}
        where
        \[ 
            \psi(a) = \lfloor a \rfloor + \min\left\{1, \frac{f(a)}{f(\tilde b)}\right\}.
        \]
    Complementing back gives us (9).
    Finally, we can use the fact that for $j \notin P$, $x_j$ is fixed to 0 and that $\psi(-a_j) \leq 0$ to show that $\reasoncMIR$ propagates $x_r$ tightly. The propagated bound is
    \begin{align*}
      1  - \sum_{j \in P} \psi(-a_j) - \max\Big\{&\sum_{j \notin P} \psi(a_j) x_j - \sum_{j \in P} \psi(-a_j) x_j\Big\} \\
            = 1  - \sum_{j \in P} \psi(-a_j)  -  \Big( &\sum_{j \notin P} \psi(a_j)\cdot 0 - \sum_{j \in P} \psi(-a_j) \cdot 1 \Big)   
            = 1. 
    \end{align*}
\end{proof}
\begin{remark}
    The reduced constraint $\reasoncMIR$ is precisely the Gomory mixed integer cut \cite[]{Gomory1960TR} generated from the optimal tableau
    obtained by solving the one-row LP \eqref{eq:1-row}.
\end{remark}

\section{Dominance Relationships}
\label{sec:dominance}

A natural goal is to find a reduction technique that yields the strongest
possible reduced constraint to use in the resolution step of conflict analysis.
In this section, we discuss dominance
relationships between the different reduction techniques in the following sense.
\begin{definition}
    A constraint $C''$ \textbf{dominates} a constraint $C'$ if any $\tilde{x}
    \in [\lb{},\ub{}]$ that satisfies $C''$ also satisfies $C'$. In other
    words, the set of feasible points defined by the variables bounds and $C''$
    is a subset of the set of feasible points defined by the variables bounds
    and $C'$.
\end{definition}

The main goal of this section is to compare the cMIR-based reduction from \Cref{sec:MIR} with the wMIR-based reduction from \cite[]{MBGN23ImprovingConflictAnalysisMIP}, which we recall next.
Let $\reasoncon: x_r + \sum_j a_j x_j \geq b$ be as in Assumption~\ref{ass:reasonForm}, propagating $x_r$ non-tightly.
Then the reduced constraint
\begin{equation} \label{eq:wMIR}
 \reasonwMIR : x_r + \sum_{j \in P_Z} a_j x_j + \sum_{j \notin P} \psi_w(a_j) x_j \geq \left\lceil b - \sum_{j \in P_W} a_j\right\rceil 
\end{equation}
propagates $x_r$ tightly, where $P = \{j : u_j = 1\}$, $P_W = \{j \in P : a_j \notin \mathbb{Z}\}$,
$P_Z = \{j \in P : a_j \in \mathbb{Z}\}$, and 
\[ 
    \psi_w(a) = \lfloor a \rfloor + \min\left\{1, \frac{f(a)}{f(b - \sum_{j \in P_W} a_j)}\right\}.
\]
The above formula comes from weakening the variables in $P_W$ and then applying the MIR cut to the resulting constraint.
In what follows we show that this is not the strongest possible reduction.
\begin{proposition}
    \label{prop:complementation-domination2}
    Let $\state{}$ be the current state and $\reasoncon: x_r + \sum_{j \in J} a_j x_j \geq b$
    be as in Assumption~\ref{ass:reasonForm},
    then $\reasoncMIR$~\eqref{eq:cMIR} dominates $\reasonwMIR$~\eqref{eq:wMIR}.
\end{proposition}
\begin{proof}

    Let $P = \{j \in J : u_j^\rho = 1\}$, $P_W = \{j \in P : a_j \notin \mathbb{Z}\}$, and
    $P_Z = \{j \in P : a_j \in \mathbb{Z}\}$.
    The constraint $\reasoncMIR$ is given by
    \[
        x_r + \sum_{j \notin P} \psi_c(a_j) x_j - \sum_{j \in P} \psi_c(-a_j) x_j \geq 1 - \sum_{j \in P} \psi_c(-a_j),
    \]
    where 
    \[ 
        \psi_c(a) = \lfloor a \rfloor + \min\left\{1, \frac{f(a)}{f(b - \sum_{j \in P} a_j)}\right\}.
    \]
    This can be rewritten as
    \[
        x_r + \sum_{j \notin P} \psi_c(a_j) x_j - \sum_{j \in P_Z} \psi_c(-a_j) x_j - \sum_{j \in P_W} \psi_c(-a_j) x_j \geq 1 - \sum_{j \in P} \psi_c(-a_j) 
    \]
    For $j \in P_Z$, we have $\psi_c(-a_j) = -a_j$.
    Therefore, we can rewrite $\reasoncMIR$ as
    \[
        x_r + \sum_{j \notin P} \psi_c(a_j) x_j + \sum_{j \in P_Z} a_j x_j - \sum_{j \in P_W} \psi_c(-a_j) x_j \geq 1 - \sum_{j \in P_W} \psi_c(-a_j) + \sum_{j \in P_Z} a_j.
    \]
    This is equivalent to
    \[
        x_r + \sum_{j \notin P} \psi_c(a_j) x_j + \sum_{j \in P_Z} a_j x_j + \sum_{j \in P_W} \psi_c(-a_j) \bar x_j \geq 1 + \sum_{j \in P_Z} a_j.
    \]
    Given that for $j \in P_W, \psi_c(-a_j) \leq 0$, the constraint
    \begin{equation} \label{eq:weakenCMIR}
        x_r + \sum_{j \notin P} \psi_c(a_j) x_j + \sum_{j \in P_Z} a_j x_j \geq 1 + \sum_{j \in P_Z} a_j,
    \end{equation}
    is dominated by $\reasoncMIR$.
    Hence, it suffices to show that $\reasonwMIR$ is equivalent to \eqref{eq:weakenCMIR}.
    To this end, notice that $\reasonwMIR$ is given by
    \[
        x_r + \sum_{j \in P_Z} a_j x_j + \sum_{j \notin P} \psi_w(a_j) x_j \geq \left\lceil b - \sum_{j \in P_W} a_j\right\rceil 
    \]
    with
    \[ 
        \psi_w(a) = \lfloor a \rfloor + \min\left\{1, \frac{f(a)}{f(b - \sum_{j \in P_W} a_j)}\right\}.
    \]
    However, the fractional parts of $b - \sum_{j \in P_W} a_j$ and $b - \sum_{j \in P} a_j$
    are equal because their difference $\sum_{j \in P_Z} a_j\in\Z$.
    That is, $f(b - \sum_{j \in P_W} a_j) = f(b - \sum_{j \in P} a_j)$.
    This implies that $\psi_w = \psi_c$.
    Finally, 
    \[
        \left\lceil b - \sum_{j \in P_W} a_j\right\rceil = \left\lceil b - \sum_{j \in P} a_j + \sum_{j \in P_Z} a_j \right\rceil =\left\lceil b - \sum_{j \in P} a_j \right\rceil + \sum_{j \in P_Z} a_j  = 1 + \sum_{j \in P_Z} a_j.
    \]
    Thus, $\reasonwMIR$ is equivalent to \eqref{eq:weakenCMIR} and, therefore, dominated by $\reasoncMIR$.
\end{proof}

For completeness, we recall that the wMIR-based reduction dominates the Chvátal-Gomory-based reduction as described in \cite[]{MBGN23ImprovingConflictAnalysisMIP}.
Hence, \Cref{prop:complementation-domination2} shows that the cMIR-based reduction also dominates the Chvátal-Gomory-based reduction.
From \citet[]{GNY19DivisionSaturation} it follows
that the Chv\'atal-Gomory and coefficient tightening reduction algorithms are incomparable. 
Similar arguments can be used to show the same result for the MIR reduction from \Cref{prop:mirReduction} and the coefficient tightening reduction from \Cref{algo:ReduceSAT}. 
Details can be found in \cite[]{MBGN23ImprovingConflictAnalysisMIP, GNY19DivisionSaturation}.

\section{Conflict Analysis for Mixed Binary Programs}
\label{sec:cut-based-conflict-analysis-for-mbp}

Our next goal is to extend cut-based conflict analysis to the class of mixed binary programs.
In \Cref{sec:reduction-techniques-for-bp} we have seen that for pure binary programs it is always possible to reduce the reason constraint to guarantee that in each iteration of conflict analysis the linear combination of the current conflict constraint and the reduced reason constraint remains infeasible under the local bounds, and hence the conflict analysis invariant is preserved. 
The key property of the reduced reason constraint is that it can be made to propagate the resolved variable tightly even when considered over the reals.
In the case of MBPs, continuous variables are always propagated tightly. 
Hence, from \Cref{prop:tight-prop-resolution} it follows that whenever resolving a continuous variable, the linear combination of the current conflict constraint and the reason constraint that propagated the variable we are resolving on remains infeasible under the local bounds.

However, as shown in detail in the following example, the mere presence of non-relaxable continuous variables can lead to a situation where resolving a binary variable $\var{\ridx}$ is impossible without using additional problem information. 

\begin{example}
\label{ex:mbp-reason-reduction}
Consider the following system of constraints:
\begin{align*}
    C_1: \, -2\var{1} - 4y_1 - 2y_2 \geq -3, \quad 
    C_2:\, 20\var{1} + 5y_1  - y_2\geq 4, \quad
    C_3:\, - 20\var{1} + 5y_1 - 10y_2\geq -16, \quad
\end{align*}
\vspace{-1cm}
\begin{align*}
    C_4:\, -y_2 -\var{2} \geq 0, \quad
    C_5:\, y_2 - \var{3} \geq 0
\end{align*}
where $\var{1}$, $\var{2}$ and $\var{3}$ are binary variables, $y_1$ is a continuous variable with global bounds $[0, 1]$, and $y_2$ is a continuous variable with global bounds $[-1, 1]$. 
After branching on $\var{2}$, in the subproblem $\var{2}=0$ we can deduce $y_2\leq 0$ from $C_4$. 
Next, from $C_5$ we can deduce $\var{3}=0$ and $y_2\geq 0$. The constraint $C_1$ propagates $y_1 \leq 3/4$ which leads to $C_2$ propagating $\var{1} \geq 1$.
        Under the local bounds, the constraint $C_3$ is infeasible.
Eliminating $\var{1}$ with $C_2$ being the reason constraint and $C_3$ the conflict, we obtain $C_2+C_3: 10y_1 - 11y_2 \geq -12$ which is not infeasible under the local bounds.
\bluediff{Next we show that it is not possible to find any globally valid inequality 
that is violated in the local bounds by considering only the constraints $C_2$ and $C_3$.
Consider the two points $p_1 = (1, 0, 0, 0, -0.4)$ and $p_2 = (0, 0, 0, 0.88,
0.4)$. Both points are integer feasible for the constraints $C_2$ and $C_3$. 
Since any globally valid inequality for the system $C_2$ and $C_3$ must satisfy the 
two points, it must also satisfy any convex combination 
$ p^\lambda = \lambda p_1 + (1-\lambda)p_2$ for $\lambda \in [0,1]$. However, the point $p^{0.5} = (0.5, 0, 0, 0.44, 0)$ 
lies in the local domain of the variables and hence cannot be separated by any globally valid inequality.
\Cref{fig:mbp-reason-reduction} shows the non-empty intersection of the mixed integer hull with the local domain.
}
\begin{figure}
    \centering
    \includegraphics[width=1.0\textwidth]{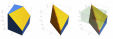}
    \caption{\bluediff{Left: Region defined by $C_2$, $C_3$, and the global domain; Middle: Mixed‐integer hull; Right: Mixed‐integer hull intersected with the local domain.}}
    \label{fig:mbp-reason-reduction}
\end{figure}
\end{example}

In the following section we propose a new reason reduction algorithm for MBP that does not only consider the reason and conflict constraint, but also constraints that propagated non-relaxable continuous variables.
\subsection{Reduction Algorithm for Mixed Binary Programs}
The goal of the reduction algorithm is to remove non-relaxable continuous variables from the reason constraint while preserving the propagation of the binary variable $\var{\ridx}$.
The next proposition shows that after resolving a non-relaxable continuous variable from the reason constraint, the new reason constraint still propagates the binary variable $\var{\ridx}$.

\begin{proposition}
    \label{prop:remove-continuous-var-from-reason2}
    Let $\contreasoncon: a_c \var{c} + a_\ridx \var{\ridx}+ \sum_{ \vidx\not\in \{ r,c \} } a_\vidx
    \var{\vidx} \geq \rhs$ be a MBP constraint propagating a continuous variable $\var{c}\in [\lb{c}, \ub{c}]$ in the state $\rho_1$.    Further, assume that the bound change on the continuous variable is non-relaxable for a constraint $\reasoncon: a'_c\var{c} + a'_\ridx \var{\ridx} + \sum_{ \vidx\not\in \{ r,c \} } a'_{\vidx} \var{\vidx}\geq b'$ that propagates a variable $\var{\ridx}$ in a later state $\rho_2$.
    Then, the linear combination of $\reasoncon$ and $\contreasoncon$ that eliminates $\var{c}$ propagates $x_r$ at least as strong in the sense that it gives the same or a better bound as $\reasoncon$ in $\rho_2$.
\end{proposition}
\begin{proof}

Since the bound change implied by $\contreasoncon$ on $x_c$ is non-relaxable for $\reasoncon$, the coefficients $a_c$ and $a'_c$ must have opposite sign, and
we can assume without loss of generality that $a_c = 1$ and $a'_c = -1$.
The bound change of $x_r$ implied by $\reasoncon$ follows from
\begin{align*}
    a'_\ridx \var{\ridx}
    \geq \min_{\rho_2}\{b' + x_c - \sum_{ \vidx \not\in \{ r,c \} } a'_{\vidx} \var{\vidx}\}
    =\ell_c^{\rho_2} + \min_{\rho_2}\{b' - \sum_{ \vidx \not\in \{ r,c \} } a'_{\vidx} \var{\vidx} \}.
\end{align*}
By assumption, $\ell_c^{\rho_2}=\ell_c^{\rho_1}$ is the result of propagating $\contreasoncon$ in state $\rho_1$ and is given by
\[
    x_c  \geq \min_{\rho_1} \{b - a_r x_r - \sum_{ j \neq \{ r,c \} } a_j  x_j \} = \ell_c^{\rho_2}.
\]
Therefore, the bound change on $x_r$ from propagating $\reasoncon$ in state~$\rho_2$ is the one implied by
\begin{equation}\label{eq:xr_by_creason}
    a'_r x_r
    \geq \min_{\rho_1} \{b - a_r x_r - \sum_{ j \neq \{ r,c \} } a_j  x_j\} + \min_{\rho_2}\{b' - \sum_{ \vidx \neq \{ r,c \} } a'_{\vidx} \var{\vidx}\}.
\end{equation}
On the other hand, the linear combination of $\reasoncon$ and $\contreasoncon$ eliminating $\var{c}$ is given by
\begin{equation}\label{eq:lincomb}
    (a_r + a'_r) x_r + \sum_{ j \neq \{ r,c \} } (a_j  + a'_j) x_j \geq b + b'.
\end{equation}
Hence, the bound change on $x_r$ from propagating this last constraint in state~$\rho_2$ is the one implied by
\begin{align}
    a'_r x_r 
    &\geq b + b' - a_r x_r - \sum_{ j \neq \{ r,c \} } (a_j  + a'_j) x_j \nonumber\\
    &= b - a_r x_r - \sum_{ j \neq \{ r,c \} } a_j x_j + b' - \sum_{ j \neq \{ r,c \} } a'_j x_j \nonumber\\
    &\geq \min_{\rho_2} \{b - a_r x_r - \sum_{ j \neq \{ r,c \} } a_j x_j \}
    + \min_{\rho_2} \{b' - \sum_{ j \neq \{ r,c \} } a'_j x_j \}. \label{eq:xr_by_lincomb} 
\end{align}
Comparing the right-hand sides of \eqref{eq:xr_by_creason} and \eqref{eq:xr_by_lincomb}, we see that the linear combination~\eqref{eq:lincomb} implies a bound at least as strong as $\reasoncon$.

\end{proof} 

\begin{remark}
    \label{rem:continuous-reduction}
    Notice that in the proof above, the bound implied by the linear combination of $\reasoncon$ and $\contreasoncon$ can be substantially stronger
    than the one implied by the reason constraint.
    Indeed, the aggregated constraint can even show infeasibility at state~$\rho_2$.
    A particularly interesting example of this situation is when the sign of the coefficient of $x_r$ in the aggregated contraint ($a_r+a'_r$)
    is different from the one in the reason constraint ($a_r'$).
    To see this, assume that $a'_r > 0$ (and that $x_r$ is binary for simplicity). Since the sign changes in the aggregated constraint, it holds that $a_r < -a'_r < 0$.
    This means that $\ell_c^{\rho_2}$ reduces to $\min_{\rho_1} \{b - \sum_{ j \neq \{ r,c \} } a_j  x_j \}$.
    Furthermore, since the reason constraint propagates $x_r$, we have that the right-hand side of \eqref{eq:xr_by_creason} is positive, i.e.,
    \[ 
      \kappa := \min_{\rho_1} \{b - \sum_{ j \neq \{ r,c \} } a_j  x_j\} + \min_{\rho_2}\{b' - \sum_{ \vidx \neq \{ r,c \} } a'_{\vidx} \var{\vidx}\} > 0.
    \]
    This gives rise to the contradiction
    \[
        0 \geq (a_r + a'_r) x_r \geq  b + b' - \sum_{ j \neq \{ r,c \} } (a_j  + a'_j) x_j \geq \kappa > 0.
    \]
    This contradiction proves infeasibility independently of the bound change on $x_r$ in state $\rho_2$.
    Hence, the aggregated constraint \eqref{eq:lincomb} is already infeasible at the state preceding $\rho_2$.  
\end{remark}

\Cref{algo:ReduceMixed} shows the reduction algorithm for MBP. 
It takes as input the reason constraint $\reasoncon$ propagating in state~$\state{\hat \nidx,\hat \bidx}$, the conflict constraint $\conflictcon$ infeasible in state $\state{\bar \nidx,\bar \bidx}$, and the binary variable to resolve~$\varr$.
Then it iteratively resolves all non-relaxable continuous variables from the reason constraint until no such variable exists.
The set of non-relaxable continuous variables 
\bluediff{at a state $\state{\nidx,\bidx}$ is denoted by $\slackreducing(\reasoncon, \state{\nidx,\bidx})$.}
The algorithm may terminate early if the reason constraint becomes infeasible under the local bounds (see Remark 3 in Appendix D).
Finally, after the main loop, we apply the reduction algorithm as in the binary case.
The algorithm always terminates since the number of states is finite and it explores the states in a monotonically decreasing order with respect to the lexicographical order of the indices $(\nidx,\bidx)$ and the non-relaxable continuous variables appearing in the reason constraint.

\newcommand{\reduceCoefTight}{\textnormal{reduceCoefTight}}
\begin{algorithm}[h]
    \DontPrintSemicolon
	\SetKwInOut{Input}{Input}\SetKwInOut{Output}{Output}
	\SetKwInOut{Init}{Initialization}
    \SetKwFunction{pop}{pop()}
	\Input{reason constraint $\reasoncon$ propagating in state $\state{\hat \nidx,\hat \bidx}$,\\
      conflict constraint $\conflictcon$, binary variable to resolve $\varr$}
	\Output{reduced reason $\reasoncon$ or an earlier conflict $\conflictcon$} 
    
    $(\tilde{\nidx},\tilde{\bidx}) \leftarrow (\hat \nidx,\hat \bidx)$\;
    \While{\bluediff{$\slackreducing(\reasoncon, \state{\tilde{\nidx},\tilde{\bidx}}) \neq \varnothing$} }
    {
        $(\tilde{\nidx},\tilde{\bidx}) \leftarrow \max \{(\nidx,\bidx) \, | \, \bdcvar(\state{\nidx,\bidx}) \in $ $\slackreducing(\reasoncon, \state{\nidx,\bidx}) \cap K , \, (\nidx,\bidx) \lexlt (\tilde{\nidx},\tilde{\bidx})  \}$\;
        $ c\leftarrow \bdcvar(\state{\tilde{\nidx},\tilde{\bidx}})$\;
        $\contreasoncon \leftarrow \reasoncons(\state{\tilde{\nidx},\tilde{\bidx}})$\;
        $\reasoncon \leftarrow \resolve(\reasoncon,\contreasoncon,\var{c})$\;
        \If{$\reasoncon$ is infeasible in the predecessor state of $\state{\hat \nidx,\hat \bidx}$}
        {
            $\conflictcon \leftarrow \reasoncon$\;
            \Return{$\conflictcon$}\;
        }
    }
    $\reasoncon \leftarrow \textnormal{reduce}(\reasoncon, \conflictcon, \varr, \state{\hat \nidx,\hat \bidx})$\;
    \Return{$\reasoncon$}\;

	\caption{Reduction Algorithm for Mixed Binary Programs \label{algo:ReduceMixed}}
\end{algorithm}

\begin{example}
    In \Cref{ex:mbp-reason-reduction} we show that no globally valid inequality can be constructed that is violated in the local bounds by using only the reason constraint~$C_2$, the conflict constraint~$C_3$ and the global bounds of the variables. Next we show that, after resolving all non-relaxable continuous variables from the reason constraint~$C_2$, the new reason constraint contains only binary variables and propagates $\var{1}$:
    \begin{enumerate}
    \item Resolving $y_1$ by adding $1.25\cdot C_1$ to $C_2$ gives the
        new $\reasoncon:  17.5\var{1} -3.5y_2 \geq 0.25$.
    \item Resolving $y_2$ by adding $3.5\cdot C_5$ to $\reasoncon$ gives the
        new $\reasoncon: 17.5\var{1} - 3.5\var{3} \geq 0.25$.
    \end{enumerate}  
        This constraint contains only binary variables, propagates $\var{1}$ to 1 non-tightly, and hence can be reduced as shown in \Cref{sec:reduction-techniques-for-bp}.
\end{example}

\section{Conflict Analysis for Mixed Integer Programs}
\label{sec:general-integer-variables}
Next, we briefly discuss the limitations of the conflict analysis algorithm
under the presence of general integer variables. Again, the goal is to resolve
a bound change on a variable by some suitable reduction and aggregation of
constraints. 
As expressed in \Cref{prop:tight-prop-resolution}, whenever the propagation of
the resolved variable is tight, we can simply aggregate the reason
for the bound change to the current conflict and the
resolvent remains infeasible.
However, in the presence of general integer variables, our reduction algorithms are not
guaranteed to find a reduced reason that propagates tightly.

Non-tight propagation occurs because of the existence of non-integer vertices for the linear relaxation of $\reasoncon$
over the local domain $L$.
The goal then is to find a cut that separates the non-integer vertex from the feasible region $\reasoncon \cap G$, where $G$ is the box defined by the global bounds. 
This is not possible if the general integer variable contributes with a non-global bound to the propagation of the reason constraint.
\Cref{fig:general-integer-cut-not-possible} gives an example that illustrates such a situation.

\begin{figure}[h]
    \centering
	\begin{tikzpicture}[scale=1.4,   every node/.style={draw, minimum size=0.15cm} 
]
		\draw[] (1,0) -- (5,1);

		\fill[blue!15] (1,0) -- (5,1) -- (5,0) -- cycle;
		\draw[dashed] (2,-0.4) -- (2,1.4);
		\draw[dashed] (3,-0.4) -- (3,1.4);
		\fill[blue!50] (2,0) -- (3,0) -- (3,0.5) -- (2,0.25) -- cycle;

		\foreach \x in {1, 2, 3, 4, 5} {
			\node[fill, circle, inner sep=1pt] at (\x, 0) {};
		}
		\foreach \x in {1, 2, 3, 4, 5} {
			\node[fill, circle, inner sep=1pt] at (\x, 1) {};
		}

		\node[fill=mygreen, circle, inner sep=1pt] at (1, 0) {};
		\node[fill=mygreen, circle, inner sep=1pt] at (2, 0) {};
		\node[fill=mygreen, circle, inner sep=1pt] at (3, 0) {};
		\node[fill=mygreen, circle, inner sep=1pt] at (4, 0) {};
		\node[fill=mygreen, circle, inner sep=1pt] at (5, 0) {};
		\node[fill=mygreen, circle, inner sep=1pt] at (5, 1) {};
		\node[fill=red, circle, inner sep=1pt] at (3,0.5) {};

		\end{tikzpicture}
	\caption{Example where it is impossible to find a linear cut that separates the non-integer vertex (2,0.5) from the feasible region.}
    \label{fig:general-integer-cut-not-possible}
\end{figure}
One approach to addressing the resolution of general integer variables is proposed in \cite[]{JdM13Cutting}.
Here the authors allow only fixings of variables to their current lower or upper bound.
While this strategy restricts the branching heuristic, it also avoids scenarios like the one illustrated in \Cref{fig:general-integer-cut-not-possible}, allowing reduction algorithms to effectively eliminate non-integer vertices.

Another work that deals with general MIP is \cite[]{Achterberg07ConflictAnalysis}, which extends the graph-based conflict analysis typically used for SAT problems to the general MIP domain.
In the MIP context, conflict graph nodes represent bound changes rather than variable fixings.
Unlike in SAT, where conflicts derived from the graph are disjunctions of literals and can be represented as linear constraints, conflicts involving non-binary variables are disjunctions of general bound changes, which cannot be captured by a single linear constraint.
Consequently, MIP solvers utilize these conflicts solely for propagation and not as cuts.

In our implementation, we consistently attempt to resolve variable bound changes by aggregating the reason and conflict constraints. 
Preliminary experiments show that in cases where integer variables do not contribute with a global bound to the propagation of the reason constraint, resolution still succeeds in 72\% of cases, meaning the resulting resolvent remains infeasible even without reducing the reason constraint.
In the remaining 28\% of cases, we attempt to separate the non-integer vertex from the feasible region heuristically, 
by applying the general cMIR procedure \cite[]{marchand2001aggregation}. 
However, this approach is only successful about 3\% of the time.

\section{Experimental Evaluation}
\label{sec:experiments}

\newcommand{\clausalCA}{\textsc{Graph-CA}\,}
\newcommand{\coefTighteningCA}{\textsc{CoefT-CA}\,}
\newcommand{\MIRCA}{\textsc{cMIR-CA}\,}

The focus of this section is to evaluate the performance of the different conflict analysis algorithms in a MIP solver. Our experiments are designed to answer the following questions:
\begin{itemize}
    \item[-] How do different conflict analysis algorithms perform 
when being applied to the same set of infeasibilities, not only in terms of performance but also regarding the characteristics of the generated constraints, e.g., constraint type, size and potential for propagation?
    \item[-] What is the performance impact of cut-based conflict analysis in the MIP solver SCIP, and how does it integrate with other conflict analysis algorithms?
\end{itemize}

All techniques discussed in this paper have been implemented in the open-source
MIP solver SCIP 9.0 \cite[]{SCIP9}. 
\bluediff{We have uploaded the code and results to a public GitHub repository hosted 
by the INFORMS Journal on Computing \cite[]{Mexi2024}.} The experiments were conducted on a cluster
equipped with Intel Xeon Gold 5122 CPUs @ 3.60GHz and 96GB of RAM. Each run was
performed on a single thread with a time limit of two hours. 
To mitigate the effects of performance variability \citep{lodi2013performance} and to ensure a fair comparison of the different conflict analysis algorithms, we used a large and diverse set of test instances, namely the MIPLIB 2017 benchmark set \citep{MIPLIB2017},  with five permutations of each individual model.
After excluding all models that are not solvable by any setting for any of the five seeds, we obtained a total of 795 measurements per run. For the remainder of this paper, we will refer to the combination of a model and a permutation as an instance.

\subsection{Implementation details}
In our experiments, we do not use the Chvátal-Gomory-based reduction since it
is dominated by the \MIRCA reduction (and does not generalize to problems with continuous variables). 
Moreover, as demonstrated in \Cref{prop:complementation-domination2}, 
 opting for complementation over weakening results in more robust reasoning constraints. 
Therefore, we consider the \MIRCA reduction solely with complementation.

For the coefficient tightening-based reduction, we employ a single-sweep approach for variable weakening,
improving the reduction algorithm's speed by avoiding
repeated activity computations and applying only one cut per iteration. As observed in
\cite[]{MBGN23ImprovingConflictAnalysisMIP}, this does usually not lead to a
loss of information since all, or almost all, variables have to
be weakened before applying coefficient tightening to obtain a tightly
propagating reason constraint.

Finally, we decided to mitigate the weak points of cut-based conflict analysis: dealing with general integer variables or propagations that are not
explained by a single linear inequality.
If cut-based conflict analysis fails to generate a conflict in such cases, we fall back to graph-based conflict analysis.
In a preliminary experiment on \miplib 2017 we measured this fallback mechanism to occur in approximately 19\% of the conflict analysis calls.

\subsection{Comparison of Conflict Analysis Algorithms in Isolation}

In our first experiment, we compare the following three conflict analysis algorithms:
\begin{itemize}
    \item[-] \clausalCA: Graph-based conflict analysis algorithm \cite[]{Achterberg07ConflictAnalysis}.
    \item[-] \coefTighteningCA: Cut-based conflict analysis using the coefficient tightening-based reduction.
    \item[-] \MIRCA: Cut-based conflict analysis using the cMIR-based reduction. 
\end{itemize}

When aiming to compare the effect of different conflict analysis algorithms throughout a MIP solve, there is a major caveat: Usually, the first few conflict analysis calls already cause the solution path of the solve to diverge significantly (demonstrating that the analysis had an impact). Thus, the vast majority of conflict analysis calls will be on a completely different set of infeasibilities, and it is hard to impossible to say whether observations made about conflicts differing in their characteristics are structural or mostly a side effect of the path divergence.

Therefore, to answer the first question from the beginning of this section, we carefully designed an experiment that allows us
to compare different conflict analysis algorithms by different conflicts from the exact same infeasibilities throughout a complete tree search. To achieve this, we split the generation and the exploitation of the conflicts into two separate runs.

In the first run, whenever an
infeasibility is found, we generate a single conflict from the 1-FUIP; however, we do not use the conflicts for propagation, do not
apply bound changes, consider them for the branching decision, or allow any other interaction with the solving process. They are collected and ``ignored''. We call this the \textit{conflict
generation run}. 
Thus, we will traverse an identical search tree, independent of the conflict analysis algorithm used, and each algorithm will be applied to the same set of infeasibilities.

In the second run, we add the conflicts from the conflict generation run to the problem (right from the beginning of the search) and allow propagation and other interactions with the solution process.
 We refer to this run as the
\textit{conflict exploitation run}. 
This approach enables us to draw conclusions about the impact of conflicts generated from the same information. It even allows for a one-to-one comparison of conflicts from different algorithms.

In the conflict generation run, we set the same time limit as in the conflict exploitation run.
However, we try to avoid hitting this time limit, since the point at which a wall-clock time limit is hit, is non-deterministic and we may observe a different  number of conflict calls between two runs, if one of them hit the time limit slightly earlier/later.
Therefore, as additional deterministic working limits, we impose a limit of 50,000 nodes and a limit of 5,000 conflict
analysis calls.
This also avoids the conflict exploitation run using unnaturally many conflicts since SCIP typically limits the number of conflicts that are handled at the same time to a few thousand.
With these settings, only 74 of the instances hit the time limit and were consequently discarded from the results below. The rest of the instances are either solved, or hit the deterministic node or conflict analysis calls limit.
Moreover, we discard a few 
instances for which the settings reported numerically inconsistent results. 
This gives us a total of 704 instances for the comparison. While SCIP usually discards conflicts from \clausalCA that are larger than a certain threshold, we deactivated this restriction for this experiment so as not to skew the results.

\begin{table}[h]
    \centering
    \caption{Comparison of the different conflict analysis algorithms.}
    \begin{tabular}{lrrrrrrrr}
        \hline
         Setting            &  opt &   time(s) & lin.confs & confs & avg.length &   used(\%) &   bdchgs \\
        \hline
         \clausalCA     &   553 &  433.0   &   - &    96.7 &     27.5 &     31.2 &      4.9 \\
         \coefTighteningCA &   562 &  437.8  &   58.7 &  99.5 &     59.6 &     44.9 &     24.4 \\
        \MIRCA  &    562 &  436.8 &           59.0 &       99.4 &     59.7 &     45.4 &     24.2 \\
         \hline
        \end{tabular}
        \label{tab:confexperiment}
\end{table}

\Cref{tab:confexperiment} summarizes the results of the experiment for the three
different conflict analysis algorithms. The table shows the number of instances
(inst), the number of instances solved to optimality (opt), the average
solving time in seconds (time(s)), the average number of conflicts
generated (confs), the average number of non-zeros in conflicts (avg.length), the
percentage of conflicts that were used in propagation (used(\%)), and the
average number of bound changes applied (per node) by the conflicts (bdchgs). 

For the two cut-based variants, we additionally show the number of
linear conflicts generated(lin. confs); recall that for general integer
variables or certain propagators, we might fall back to graph-based conflict
analysis. The confs column, in these cases, considers both cut-based and
graph-based conflicts.

Our first observation is that both \coefTighteningCA and \MIRCA perform very similarly, which 
aligns nicely with the results for pure binary programming by \cite[]{MBGN23ImprovingConflictAnalysisMIP}. 
When compared to \clausalCA,  both solve eight more instances to optimality and the average time is comparable. 
In the following, we will concentrate on comparing \MIRCA to \clausalCA; all observations and conclusions would be the same or very similar for a \coefTighteningCA to \clausalCA comparison.

Two of the most noticeable differences in \Cref{tab:confexperiment} are the large discrepancy in the number of propagations and the percentage of conflicts actually used in propagation. Both are significantly larger for the cut-based variants, which is a favorable result.
When analyzing our results, we found three different reasons for this behavior.

Firstly,  \MIRCA conflicts are quite different from \clausalCA conflicts. On the one hand, \clausalCA conflicts are always logic clauses or bound disjunctions that only
propagate one single bound change when all other variables in the clause are fixed to a value not satisfying the clause. On the other hand, \MIRCA conflicts might propagate multiple bound changes even when most variables are not fixed.
For example, for the instance \texttt{neos-957323} with seed 0, \MIRCA generated the following set packing conflict constraint:
\[
 x_{1130} +x_{1131} +x_{1132} + x_{1133} +x_{1134}  +x_{1135} + x_{1150} +x_{1153} +x_{1156} \le 1,
\]
which propagates all remaining variables to zero whenever one of the variables in the constraint is fixed to one.

Secondly, both \MIRCA and \coefTighteningCA conflicts are, on average, twice as long as those from \clausalCA.  As an extreme example, consider the instance \texttt{nw04} with seed 0. This instance has 87\,482 variables, 36 constraints, and 636\,666 non-zeros, leading to an average of 17\,685 non-zeros per constraint. 
Conflict constraints generated from \MIRCA have an average length of 18\,160 non-zeros, thus very similar to the model constraints.
At the same time, the \clausalCA conflicts have an average length of ``only'' 243.
While clausal conflicts tend to be weaker the longer they are, conflicts from the cut-based approach can get stronger the longer they are, see the example above or the rich literature on lifting cutting planes.
Consequently, we observed a total of 864 bound changes from cut-based conflict constraints for instance \texttt{nw04} and only 26 propagations from \clausalCA conflicts.

Thirdly, there is another peculiar situation relating to general integer and continuous variables. Take as an example the instance \texttt{gen-ip002} with seed 1. For this instance all variables involved in conflict constraints are general integers with lower bound 0 and an upper bound of 140 or less. 
Conflicts from \clausalCA are clauses and are therefore restricted to one particular bound change per variable that they can propagate, in our example, those was typically tightening the upper bound to a particular single-digit value.
Conflicts from cut-based conflict analysis, however, can propagate arbitrary bound changes, in particular, the same conflict can tighten the upper (or lower) bound of the same variable several times during the same dive in the tree search or even at the same node (if other propagators tightened bounds of other variables in the conflict in the meantime).
The difference in structure, while having a similarly positive impact on performance, indicates that there might be potential for the two strategies to complement each other when applied in combination.
Hence, our next experiment addresses a proper integration of cut-based conflict analysis into default SCIP.

\subsection{MIP Performance}
 
In this experiment, we activate cut-based conflict analysis as an additional conflict analysis method within SCIP. 
This means that conflicts from both cut- and graph-based conflict analysis, are generated at each infeasibility. 
\bluediff{For the cut-based conflict analysis, we only use the cMIR-based reduction. As seen in the previous section \MIRCA and \coefTighteningCA
perform similarly, however in an extended experiment with a longer time limit of 6 hours and 10 different random seeds, the cMIR-based reduction was
on average 8\% faster on ``affected'' instances and required 8\% fewer nodes on those solved to optimality.}

SCIP discards dense conflicts and regularly removes conflicts that
haven't propagated in a while based on an aging strategy. 
We use the same length and age limits also for conflicts from cut-based conflict analysis.
Preliminary experiments showed a slowdown of 4-5\% if we do not limit the
length of conflict constraints. 
On the other hand, the number of explored nodes decreases by ~2\% on average. 

\begin{table}
    \caption{Performance comparison of \scip vs \scip+\MIRCA.}
    \label{tbl:MIP5seeds}
    \small
    \begin{tabular*}{\textwidth}{@{}l@{\;\;\extracolsep{\fill}}rrrrrrrrr@{}}
    \toprule
    &           & \multicolumn{3}{c}{\scip} & \multicolumn{3}{c}{\scip + \MIRCA} & \multicolumn{2}{c}{relative} \\
    \cmidrule{3-5} \cmidrule{6-8} \cmidrule{9-10}
    Subset                & instances &                                   solved &       time &        nodes &                                   solved &       time &        nodes &       time &        nodes \\
    \midrule
    all & \bluediff{794} & \bluediff{636} & \bluediff{445.9} & \bluediff{-} & \bluediff{644} & \bluediff{424.2} 
    & \bluediff{-} & \bluediff{0.95} & \bluediff{-} \\
    \affected             &492 & 463 & 454.1 & - & 471 & 419.2 & - & 0.92 & -\\
    \cmidrule{1-10}
   \bracket{10}{tilim}   &       598 & 569 & 406.0 & - & 577 & 380.2 & - & 0.94 & -\\
    \bracket{100}{tilim}  &       443 & 414 & 987.7 & - & 422 & 890.7 & - & 0.90 &-\\

    \bracket{1000}{tilim} &        243 & 214 & 2594.4 & - & 222 & 2217.8 & - & 0.85 & -\\

    \difftimeouts         &         50 &  21 &  5934.9 & - & 29 &  4180.8 &  - &  0.70 &  - \\
    \midrule
    \alloptimal           &    615 & 615 & 201.2 & 2412 & 615 & 194.1 & 2283 & 0.96 & 0.95 \\
    aff.-\alloptimal            &442 & 442 & 339.4 & 5413 & 442 & 323.0 & 5022 & 0.95& 0.93 \\
            \bottomrule
    \end{tabular*}
                \label{tab:mipexperiment}
    \end{table}
   
    \Cref{tab:mipexperiment} shows the performance of SCIP with and without the 
    \MIRCA on MIPLIB2017.
    \bluediff{One instance was excluded due to numerical issues, leaving 794 instances in total.}
    The table is divided into
    different subsets of instances. The first row shows the performance on all
    \bluediff{794}~instances, and the second row on all 492~instances that are ``affected''
    by conflict analysis, i.e., where the execution path of the two settings differs.
    The next three rows show the performance on instances where at least one of
    the two settings needs at least 10, 100, or 1000 seconds to solve the
    instance, respectively. 
    \bluediff{
    The next row ``\difftimeouts'' shows the performance
    on instances where only one of the two settings is able to solve the instance to optimality.
    \scip solved 21 instances that \scip+ \MIRCA did not, while  \scip+ \MIRCA solved 29 instances that \scip did not.
    In a post-hoc analysis with a time limit of 6 hours, there is only a single instance that \scip
    solved in under 2 hours that  \scip+ \MIRCA still could not solve after 6 hours.
    Notably, there are 9 instances that \scip failed to solve within 6 hours, 
    while \scip+ \MIRCA had already solved them in under 2 hours.
    }
    The last two rows show the performance on instances that are
    solved to optimality by both settings and the subset of those instances
    that are affected by conflict analysis ``aff.-\alloptimal''. We included these 
    rows because they allow for a comparison of tree sizes.
    
    The results show that using \MIRCA leads
    to a notable improvement in the performance of SCIP. Namely, it increases
    the number of solved instances from \bluediff{636} to 644 and decreases the overall
    running time by 5\%. On affected instances, we observed a runtime decrease
    of 8\%. As shown in the rows \bracket{10}{tilim},
    \bracket{100}{tilim}, and \bracket{1000}{tilim}, the performance improvement
    is more pronounced on harder instances, with a speedup of up to 15\% on
    instances that need at least 1000 seconds to solve.
    
    The performance improvement is also reflected in the number of nodes. The
    number of nodes can only be reliably compared on instances that are solved
    to optimality by both settings. In this case, the number of nodes is reduced
    by 5\% overall and by 7\% on affected instances.

\section{Conclusion}
\label{sec:conclusion}

In this work, we extended the cut-based conflict analysis algorithm, which
originates from pseudo-Boolean solving, to MIP. 
Framed in MIP terminology, conflict analysis can be understood as a
sequence of linear combinations, integer roundings, and cut generation.
This MIP perspective allowed us to interpret the reason reduction subroutine as
the separation of a non-integer vertex from the feasible region defined by the
reason constraint and global bounds.
Hence, our presentation helped to bridge the gap between pseudo-Boolean and MIP views on conflict
analysis, connecting the study of cuts for the knapsack polytope with the reason
reduction for pure binary problems.
We then proposed a new reduction algorithm for pure binary programs using mixed
integer rounding cuts and proved that it is stronger than the division-based reduction
algorithm \cite[]{EN18RoundingSat} and our previous work
\cite[]{MBGN23ImprovingConflictAnalysisMIP}.

On our path to generalize cut-based conflict analysis to mixed integer problems
in full generality we encountered both positive and negative results.
First, we showed that the cut-based approach cannot be
directly applied to mixed binary programs without considering additional problem
information.
To address this limitation, we designed a new reduction algorithm
that not only utilizes two constraints (reason and conflict) in each iteration
of conflict analysis but also incorporates constraints whose propagation is
responsible for tightened bounds of continuous variables.
We demonstrate that in the presence of general integer variables conflict
analysis cannot be guaranteed to work as long as we refrain from lifting the
problem into higher dimension.
However, our empirical results show that these cases occur only rarely in
practice, rendering our cut-based conflict analysis a fruitful approach also for
general integer variables.

Our computational study, first helped us to understand that the learned constraints generated by the
cut-based conflict analysis differ structurally from those learned by clausal
reasoning and propagate variable bounds more frequently on average.
Finally, we observed that the cut-based conflict analysis is able to significantly enhance the out-of-the-box performance 
of the MIP solver SCIP, solving more instances, reducing the running time and
the size of the search tree over a large and diverse set of MIP instances from
\miplib2017 by at least 5\%.
This result constitutes a substantial performance improvement for a single technique in the
mature field of MIP solving.

\bibliographystyle{plainnat}

\bibliography{references}

\end{document}